\documentclass{article}
\usepackage{amsmath,amsthm,amssymb,geometry,enumerate,indentfirst}
\title{Sum of Two Squares in Cyclic Quartic Fields}
\author{Wenhuan Huang}
\theoremstyle{definition}
\newtheorem{theorem}{Theorem}[section]
\newtheorem{corollary}[theorem]{Corollary}
\newtheorem{lemma}[theorem]{Lemma}
\newtheorem{example}[theorem]{Example}

\begin{document}

\maketitle

\begin{abstract}
This paper gives an algorithm to determine whether a number in a cyclic quartic field is a sum of two squares,
mainly based on local-global principle of isotropy of quadratic forms.

\end{abstract}

\section*{Introduction}

Previously, Qin gave an criterion determining whether a number in an arbitrary quadratic number field
is a sum of two squares in [4], which later helped him with K-Theory. The author intends to 
discover some useful information about cyclic quartic fields, to help with research on them
about K-theory and related problems.

[1] gives explicit description of cyclic quartic fields, and some of their important parameters:

\begin{theorem}
    $K$ is a (real or imaginary) cyclic quartic extension of $\mathbb{Q}$, if and only if $K=\mathbb{Q}(\sqrt{A(D+B\sqrt{D})})=\mathbb{Q}(\sqrt{A(D-B\sqrt{D})})$,
    where $A$ is odd and square-free, $D=B^2+C^2$ is square-free, $B>0$, $C>0$, $(A,D)=1$.    
\end{theorem}

\begin{theorem}
    The conductor of $K$ is $2^l|A|D$, where
    $$l=\left\{
        \begin{array}{lll}
            3 &   & if\ D\equiv2\pmod8, or\ D\equiv1\pmod4\ with\ B\equiv1\pmod2 \\ 
            2 &   & if\ D\equiv1\pmod4, \ B\equiv0\pmod2,\ A+B\equiv3\pmod4 \\ 
            0 &   & if\ D\equiv1\pmod4, \ B\equiv0\pmod2,\ A+B\equiv1\pmod4. 
        \end{array}
    \right.    
    $$
\end{theorem}

Denote $k=\mathbb{Q}(\sqrt{D})$ the unique quadratic subfield of $K$.

\begin{theorem}
    Let $\Delta$ denote the (relative) discriminant,  then
    $$\Delta_{K/k}=\left\{
        \begin{array}{lll}
            4A\sqrt{D} &   & if\ B\equiv0\pmod2\ ,\ A+B\equiv3\pmod4 \\ 
            A\sqrt{D} &   & if\ D\equiv1\pmod4, \ B\equiv0\pmod2,\ A+B\equiv1\pmod4 \\ 
            8A\sqrt{D} &   & if\ D\equiv1\pmod4, \ B\equiv1\pmod2. 
        \end{array}
    \right.    
    $$
    $$\Delta_{K}=\left\{
        \begin{array}{lll}
            2^8A^2D^3 &   & if\ D\equiv2\pmod8 \\
            2^4A^2D^3 &   & if\ D\equiv1\pmod4, \ B\equiv0\pmod2,\ A+B\equiv3\pmod4 \\ 
            A^2D^3 &   & if\ D\equiv1\pmod4, \ B\equiv0\pmod2,\ A+B\equiv1\pmod4 \\ 
            2^6A^2D^3 &   & if\ D\equiv1\pmod4, \ B\equiv1\pmod2. 
        \end{array}
    \right.    
    $$
\end{theorem}

[2, 66:1] tells the isotropy of regular quadratic forms satisfies local-global principle:

\begin{theorem}
    A regular quadratic space over a global field is isotropic if and only if it is isotropic at all spots.
\end{theorem}

Thus, a non-zero $m\in K$ is a sum of squares, if and only if $x^2+y^2-mz^2=0$ is isotropic at every spot of $K$, since $\sqrt{-1}\notin
K$ implies the solution has $z$-component not 0.

For convenience, we denote $\sigma$ the generator of $Gal(K/\mathbb{Q})$ satisfying
$$\sqrt{A(D+B\sqrt{D})}\stackrel{\sigma}{\longrightarrow}\sqrt{A(D-B\sqrt{D})}\stackrel{\sigma}{\longrightarrow}-\sqrt{A(D+B\sqrt{D})}\stackrel{\sigma}{\longrightarrow}-\sqrt{A(D-B\sqrt{D})},$$
$$\sqrt{D}\stackrel{\sigma}{\longrightarrow}-\sqrt{D}.$$

\section{Prime numbers}
First we introduce some notations. Let $p$ be a prime number. 

If $pO_k$ ramifies into $\mathfrak{p}^2$ and $\mathfrak{p}O_K$ ramifies into
$\mathcal{P}^2$, we call the field $K$ is of type RR($p$), and also denote RR($p$) the set consisting of all RR($p$)-type fields.

If $pO_k$ splits into $\mathfrak{p}_1\mathfrak{p}_2$ and $\mathfrak{p}_1O_K,\mathfrak{p}_2O_K$ inert(resp. split into $\mathcal{P}_{11}\mathcal{P}_{12}$ and $\mathcal{P}_{21}\mathcal{P}_{22}$, and 
ramify into $\mathcal{P}_1^2$ and $\mathcal{P}_2^2$), 
we call the field $K$ is of type SI($p$)(resp. SS($p$) and SR($p$)), and also denote SI($p$)(resp. SS($p$) and SR($p$)) the set consisting of all SI($p$)(resp. SS($p$) and SR($p$))-type fields.

If $pO_k$ inerts and $pO_K$ inerts(resp. ramifies into $\mathcal{P}^2$), we call the field $K$ is of type II($p$)(resp. IR($p$)), and also denote II($p$)(resp. IR($p$)) the set consisting of all II($p$)(resp. IR($p$))-type fields.

For every field $K$ and every prime $p$, $K$ must be of one of above six types.

Let us first determine if the prime number $p$ is the sum of two squares in $K$. 
If $p\equiv1,2\pmod4$, then $p$ is always the sum of two squares in $\mathbb{Q}$. Assume $p\equiv3\pmod4$.
By [2, 63:12], we just need to consider dyadic and $p$-adic primes. First we can obtain two lemmas about quadratic extensions on $\mathbb{Q}_2$ by direct calculation:

\begin{lemma}
    Let $\mathbb{Q}_2(\sqrt{c})$ be a quadratic extension on $\mathbb{Q}_2$, $c\in\{2,3,7,10,11,14,5\}$, then $t\in \mathbb{Q}_2(\sqrt{c})$ is 
    a square, if and only if

    ($c=2$)$t=2^mr$, $m\geq0$ an integer, $\sqrt{2}\nmid r$, $r\equiv 1, 3+2\sqrt{2}\pmod{4\sqrt{2}}$.

    ($c=3$)$t=(\sqrt{3}-1)^{2m}r$, $m\geq0$ an integer, $(\sqrt{3}-1)\nmid r$, $r\equiv 1, 3\pmod{4(\sqrt{3}-1)}$.

    ($c=7$)$t=(3-\sqrt{7})^{2m}r$, $m\geq0$ an integer, $(3-\sqrt{7})\nmid r$, $r\equiv \pm1\pmod{4(3-\sqrt{7})}$.

    ($c=11$)$t=(\sqrt{11}-3)^{2m}r$, $m\geq0$ an integer, $(\sqrt{11}-3)\nmid r$, $r\equiv 1,3\pmod{4(\sqrt{11}-3)}$.

    ($c=14$)$t=(4-\sqrt{14})^{2m}r$, $m\geq0$ an integer, $(4-\sqrt{14})\nmid r$, $r\equiv \pm1\pmod{4(4-\sqrt{14})}$.

    ($c=10$)$t=Mr$, $v_{(2,\sqrt{10})}(M)$ is even, $v_{(2,\sqrt{10})}(r)=0$, $r\equiv 1, 3+2\sqrt{10}\pmod{4(2,\sqrt{10})}$.

    ($c=5$)$t=4^{m}r$, $m\geq0$ an integer, $2\nmid r$, $r\equiv 1,\frac{3\pm\sqrt{5}}{2}\pmod4$.
\end{lemma}

\begin{proof}
    We take $c=2$ as example. Let $\mathfrak{p}=(\sqrt{2})$ the prime ideal of 
    $\mathbb{Q}_2(\sqrt{2})$, and $q$ a unit in $O_{\mathbb{Q}_2(\sqrt{2})}$.
    We have $q\equiv1\pmod\sqrt{2}$, since $O_{\mathbb{Q}_2(\sqrt{2})}/\mathfrak{p}\simeq \mathbb{Z}/2\mathbb{Z}$.
    So $q\equiv1,1+\sqrt{2}\pmod{2}$, implying $q^2\equiv1,3+2\sqrt{2}\pmod{4\sqrt{2}}$
    because $(q+2)^2-q^2=4(q+1)\equiv0\pmod8$.
    Conversely, if $r=1,3+2\sqrt{2}$, then $q^2\equiv r\pmod{4\sqrt{2}}$ is solvable in $O_{\mathbb{Q}_2(\sqrt{2})}$..
    Let $f(x)=x^2-r$, then $f'(x)=2x$. Since $q^2\equiv r\pmod{4\sqrt{2}}$, we have $v_\mathfrak{p}(f(q))\geq5$
    and $v_\mathfrak{p}(f'(q))=2$. By Hensel's Lemma, $f(x)$ has a zero in $O_{\mathbb{Q}_2(\sqrt{2})}$.

    The proof in other 6 situations are similar.
\end{proof}

\begin{lemma}
    Take assumptions as the last lemma and let $h$ be a unit in $O_{\mathbb{Q}_2(\sqrt{c})}$, then $h$ is the sum of two squares if and only if 
    
    ($c\neq 5$)$2|h-1$.

    ($c=5$)$h\equiv1,3,\frac{\pm3\pm\sqrt{5}}{2}\pmod4$.
\end{lemma}

By the lemmas above, since $p\equiv3\pmod4$, the Hilbert symbol $(\frac{-1,p}{\mathcal{P}})$ where $P$ is a dyadic spot of $K$, if and only if 
$K_\mathcal{P}\not\simeq\mathbb{Q}_2$, i.e, $K\notin SS(2)$. (For definition and properties of Hilbert symbol, see Section 63B of [2].)

Next determine whether $(\frac{-1,p}{\mathcal{P}})$ where $P$ is a p-adic spot of $K$ is 1. If $K\in SS(p)$ then $K_\mathcal{P}\not\simeq\mathbb{Q}_p$.
If $x^2+y^2=p$ is solvable then $(\frac{-1}{p})=1$, contradicts with $p\equiv3\pmod4$. 
Otherwise, $K_\mathcal{P}$ contains a quadratic extension of $\mathbb{Q}_p$.

\begin{lemma}
    (1)-1 is the sum of two squares in $K$, if and only if $A<0$ and $K\notin SS(2)$.

    (2)If $p\equiv3\pmod4$, $x^2+y^2=p$ is solvable in $K_\mathcal{P}$ if and only if $K\notin SS(2)$ and $K\notin SS(p)$.
\end{lemma}

\begin{proof}
    (1)Directly obtained from [2, 63:12] and Lemma 1.2.

    (2)$K_\mathcal{P}$ contains a quadratic extension $\mathbb{Q}_p(\sqrt{c})$ of $\mathbb{Q}_p$. By $p\equiv3\pmod4$ we can assume $c\in\{-1,p,-p\}$.
    For $c=p$ take $x=\sqrt{p}$ and $y=0$. For $c=-p$, by (1) let $u^2+v^2=-1$ in $K$, and take $x=u\sqrt{-p}$ and $y=v\sqrt{-p}$.
    For $c=-1$, $x^2+(\sqrt{-1}y)^2=p$ is of course solvable.
\end{proof}

\begin{corollary}
    For any prime number $p\equiv3\pmod4$, $p$ is a sum of two squares in $K$ if and only if $K\notin SS(2)$ and $K\notin SS(p)$.
\end{corollary}

\section{General Cases}

Let $\theta=\sqrt{A(D+B\sqrt{D})}$, $m=X+Y\theta$, $X=x_1+x_2\sqrt{D}$, $Y=y_1+y_2\sqrt{D}$, $x_1,x_2,y_1,y_2\in\mathbb{Z}$. Furthermore,
we assume $(x_1,x_2,y_1,y_2)=1$.
 Then we have 
$$N_{K/k}(m)=x_1^2+Dx_2^2+AD(y_1^2+Dy_2^2+2By_1y_2)+\sqrt{D}(2x_1x_2+2ADy_1y_2+AB(y_1^2+Dy_2^2)),$$
$$N_{K/\mathbb{Q}}(m)=(x_1^2+Dx_2^2+AD(y_1^2+Dy_2^2+2By_1y_2))^2-{D}(2x_1x_2+2ADy_1y_2+AB(y_1^2+Dy_2^2))^2.$$

Let $\sigma$ be the generator of $Gal(K/\mathbb{Q})$. First, to make $m$ is the sum of two squares, at least at infinite spots, we need
: If $A>0$, $m,\sigma(m),\sigma^2(m),\sigma^3(m)>0$.

If $p\equiv1\pmod4$, then $x^2=-1$ is solvable in $\mathbb{Z}/p\mathbb{Z}$, and therefore in $\mathbb{Q}_p$.
Hence $(\frac{-1,*}{K_\mathcal{P}})$ is always 1.
So we just need to determine all $(\frac{-1,m}{\mathcal{P}})$s with $\mathcal{P}|2N_{K/\mathbb{Q}}(m)$ and $p\not\equiv1\pmod4$.

Unless specifically claimed, for $\mathcal{P}$ a spot of $K$, denote $\mathfrak{p}=\mathcal{P}\cap k$, $p=\mathcal{P}\cap \mathbb{Q}$.
First we assume $p\neq2$ and $p|N_{K/\mathbb{Q}}(m)$.

(1)$p|D$, i.e, $K\in RR(p)$. Then $p|N_{K/\mathbb{Q}}(m)$ implies $p|x_1$, equivalent that $\mathfrak{p}=(p,\sqrt{D})|N_{K/k}(m)$.
If $\mathcal{P}^2|m$, then $p^2|N_{K/\mathbb{Q}}(m)$, so $p|y_1$. If $\mathcal{P}^3|m$, then $p^3|N_{K/\mathbb{Q}}(m)$, which implies
$p|x_1, p|y_1, p|x_2$. $\mathcal{P}^4|m$ implies $p|(x_1,y_1,x_2,y_2)$, contradicts with the assumption. 

Hence $v_\mathcal{P}(m)=2$ if $p|x_1, p|y_1, p\nmid x_2$, where $(\frac{-1,m}{\mathcal{P}})=1$ always holds. Otherwise, 
$v_\mathcal{P}(m)=1$ or 3, where $(\frac{-1,m}{\mathcal{P}})=1$ is equivalent that -1 is a square in $K_\mathfrak{p}$, i.e, 
$\sqrt{-1}\in K_\mathcal{P}=k_\mathfrak{p}(\theta)$, i.e, $-A(D+B\sqrt{D})\in \mathbb{Q}_p(\sqrt{D})^{*2}$, which is impossible since 
$v_\mathfrak{p}(-A(D+B\sqrt{D}))=1$.

Hence we obtain that

\begin{lemma}
    If $(x_1,x_2,y_1,y_2)=1$, $p\equiv3\pmod4$ and $p|(D,N_{K/\mathbb{Q}}(m))$, then $(\frac{-1,m}{\mathcal{P}})=1$ if $p|x_1, p|y_1, p\nmid x_2$, and -1 otherwise.
\end{lemma}

(2)$pO_k$ inerts.

If $K\in II(p)$, the assumption decides $v_p(m)=0$, contradicts the assumption that $p\nmid m$ and $p|N_{K/\mathbb{Q}}(m)$.

If $K\in IR(p)$, the assumption implies $v_\mathcal{P}(m)=1$.
Since $p\equiv3\pmod4$ we have $k_\mathfrak{p}\simeq\mathbb{Q}_p(\sqrt{-1})$, where $-1$ is always a square.
Hence $(\frac{-1,m}{\mathcal{P}})=1$. 

\begin{lemma}
    If $p\equiv3\pmod4$, $(\frac{D}{p})=-1$, then $(\frac{-1,m}{\mathcal{P}})=1$.
\end{lemma}

(3)$pO_k$ splits into $\mathfrak{p}_1\mathfrak{p}_2$, i.e, $(\frac{D}{p})=1$,
then $k_{\mathfrak{p}_1}\simeq k_{\mathfrak{p}_2}\simeq\mathbb{Q}_p$.

If $K\in SS(p)$, with $\mathfrak{p}_1=\mathcal{P}_{11}\mathcal{P}_{12}$ and $\mathfrak{p}_2=\mathcal{P}_{21}\mathcal{P}_{22}$,
then there are more possible cases.

Case A. If both $m\sigma(m)$ and $m\sigma^{-1}(m)$ are divided by $p$, then one and only one of $\mathfrak{p}_1$ and $\mathfrak{p}_2$
divides $m$. Without loss of generality, assume $\mathfrak{p}_1|m$, and 
$v_{\mathcal{P}_{11}}(m)\geq v_{\mathcal{P}_{12}}(m)=v_{\mathfrak{p}_1}(m)$,
$v_{\mathcal{P}_{21}}(m)\geq v_{\mathcal{P}_{22}}(m)=0$,
then we have 
$$v_{\mathcal{P}_{21}}(m)=v_{\mathfrak{p}_{2}}(m\sigma^2(m))=v_{\mathfrak{p}_{2}}(N_{K/k}(m)),$$
$$v_{\mathcal{P}_{11}}(m)=v_p(N_{K/\mathbb{Q}}(m))-0-v_{\mathfrak{p}_1}(m)-v_{\mathfrak{p}_{2}}(N_{K/k}(m)),$$
Hence we have: $(\frac{-1,m}{\mathcal{P}_{ij}})=1$ for $i,j\in\{1,2\}$ if and only if 
$v_{\mathfrak{p}_1}(m)(:=min\{v_{\mathfrak{p}_1}(X),v_{\mathfrak{p}_1}(Y)\})$, $v_{\mathfrak{p}_{2}}(N_{K/k}(m))$ and $v_p(N_{K/\mathbb{Q}}(m))$ are even, otherwise
$(\frac{-1,m}{\mathcal{P}_{ij}})$s are distinct.

Case B. Only one of $m\sigma(m)$ and $m\sigma^{-1}(m)$ is divided by $p$, then neither $\mathfrak{p}_1$ nor $\mathfrak{p}_2$
divides $m$. Suppose $p|m\sigma(m)$, then $v_p(m\sigma(m))$ is exactly the lower one of the two 
non-zero $v_{\mathcal{P}_{ij}}$s, with the higher one $v_p(N_{K/\mathbb{Q}}(m))-v_p(m\sigma(m))$.
Hence we have: $(\frac{-1,m}{\mathcal{P}_{ij}})=1$ for $i,j\in\{1,2\}$ if and only if
both $v_p(m\sigma(m))$ and $v_p(N_{K/\mathbb{Q}}(m))$ are even, otherwise $(\frac{-1,m}{\mathcal{P}_{ij}})$s are distinct.

Case C. Neither $m\sigma(m)$ nor $m\sigma^{-1}(m)$ is divided by $p$, then the only non-zero $v_{\mathcal{P}_{ij}}(m)$ is 
exactly $v_p(N_{K/\mathbb{Q}}(m))$. So $(\frac{-1,m}{\mathcal{P}_{ij}})=1$ for $i,j\in\{1,2\}$ if and only if
$v_p(N_{K/\mathbb{Q}}(m))$ is even, otherwise $(\frac{-1,m}{\mathcal{P}_{ij}})$s are distinct.

For convenience we express

Condition $[m, p]$: 
(1)If both $m\sigma(m)$ and $m\sigma^{-1}(m)$ are divided by $p$ with $\mathfrak{p}_1|m$, 
$v_{\mathfrak{p}_1}(m)=min\{v_{\mathfrak{p}_1}(X),v_{\mathfrak{p}_1}(Y)\}$, $v_{\mathfrak{p}_{2}}(N_{K/k}(m))$ and $v_p(N_{K/\mathbb{Q}}(m))$ are even;

(2)If only $m\sigma(m)$(resp. $m\sigma^{-1}(m)$) is divided by $p$, then both $v_p(m\sigma(m))$(resp. $m\sigma^{-1}(m)$)
and $v_p(N_{K/\mathbb{Q}}(m))$ are even;

(3)If neither $m\sigma(m)$ nor $m\sigma^{-1}(m)$ is divided by $p$, $v_p(N_{K/\mathbb{Q}}(m))$ is even.

\begin{lemma}
    Assume $p\equiv3\pmod4$. If $(\frac{D}{p})=1$ and $(\frac{A(D+Bc)}{p})=1$ where $c$ is an integer that $c^2\equiv D\pmod p$,
    then
    
    (1)$(\frac{-1,m}{\mathcal{P}})=1$ for all $p$-adic places $\mathcal{P}$, if and only if condition $[m, p]$ holds;

    (2)Otherwise, $(\frac{-1,m}{\mathcal{P}})$s are distinct.

\end{lemma}

If $K\in SI(p)$, then $K_{\mathcal{P}_1}\simeq K_{\mathcal{P}_2}$ is unramified on $\mathbb{Q}_p$, i.e, 
$K_{\mathcal{P}_1}\simeq K_{\mathcal{P}_2}\simeq\mathbb{Q}_p(\sqrt{-1})$ since $p\equiv3\pmod4$.
Hence $-1$ is a square in $K_{\mathcal{P}_1}\simeq K_{\mathcal{P}_2}$. 
So $(\frac{-1,m}{\mathcal{P}_1})=(\frac{-1,m}{\mathcal{P}_2})=1$.

If $K\in SR(p)$, i.e, $p|A$, $\mathfrak{p}_1O_K=\mathcal{P}_1^2$ and $\mathfrak{p}_2O_K=\mathcal{P}_2^2$, then 
$K_{\mathcal{P}_1}\simeq K_{\mathcal{P}_2}$
is quadratic ramified extension on $\mathbb{Q}_p$, i.e, 
$K_{\mathcal{P}_1}\simeq K_{\mathcal{P}_2}\simeq \mathbb{Q}_p(\sqrt{p})$ or $\mathbb{Q}_p(\sqrt{-p})$,
depending on whether $(\frac{(A(D+Bc))/p}{p})$ is 1, or not(where $c$ is an integer that $c^2\equiv D\pmod p$). 
Hence $-1$ is a non-square in $K_{\mathcal{P}_1}\simeq K_{\mathcal{P}_2}$, and then
$(\frac{-1,m}{\mathcal{P}_i})=(-1)^{v_{\mathcal{P}_i}(m)}$, $i=1,2$.
To calculate $v_{\mathcal{P}_i}(m)$, without loss of generality, assume that $v_{\mathcal{P}_1}(m)\geq v_{\mathcal{P}_2}(m)$, 
then $v_{\mathcal{P}_2}(m)$ must be 1, if $p|m^2$; or 0, if not.
If $p\nmid m^2$, then $v_{\mathcal{P}_1}(m)=v_{p}(N_{K/\mathbb{Q}}(m))$.
If $p|m^2$, then $v_{\mathcal{P}_1}(m)=v_{p}(N_{K/\mathbb{Q}}(m))-1$.
Especially, for $0\neq M\in\mathbb{Q}$, $v_{\mathcal{P}_1}(M)$ is always even.

\begin{lemma}
    Assume $p\equiv3\pmod4$. 
    If $(\frac{D}{p})=1$ and $p|A$, then     
    $(\frac{-1,m}{\mathcal{P}})=1$ for every $p$-adic spots $\mathcal{P}$, if and only if
    $p\nmid m^2$, and $v_{p}(N_{K/\mathbb{Q}}(m))$ is even.
\end{lemma}

In conclusion, the arguments above can be used to determine all $(\frac{-1,m}{\mathcal{P}})$s at all non-dyadic spots $\mathcal{P}$.

Finally, we investigate dyadic cases. 
Let $\mathcal{P}$ be a dyadic spot of $K$.
If $2O_k$ does not split, then $K$ has only one dyadic spot, and by Hilbert Reciprocity Law,
$x^2+y^2-mz^2=0$ is isotropic at this spot whereas $m\neq0$. 
Otherwise $D\equiv 1\pmod8$, $[K_\mathcal{P}:\mathbb{Q}_p]$ is at most 2.
We employ lemma 1.2 and 1.3. Let 
$$e(w)=\left\{
    \begin{array}{lll}
        1 &   & if\ w\equiv 1^*\\
        3 &   & if\ w\equiv 9^*\\
        5 &   & if\ w\equiv 25^*\\
        7 &   & if\ w\equiv 49^*\\
        9 &   & if\ w\equiv 81^*\\
        11 &   & if\ w\equiv 121^*\\
        13 &   & if\ w\equiv 169^*\\
        15 &   & if\ w\equiv 225^*\\
        17 &   & if\ w\equiv 33^*\\
        19 &   & if\ w\equiv 105^*\\
        21 &   & if\ w\equiv 185^*\\
        23 &   & if\ w\equiv 17^*\\
        25 &   & if\ w\equiv 113^*\\
        27 &   & if\ w\equiv 217^*\\
        29 &   & if\ w\equiv 73^*\\
        31 &   & if\ w\equiv 193^*\\
        33 &   & if\ w\equiv 65^*\\
        35 &   & if\ w\equiv 201^*\\
        37 &   & if\ w\equiv 89^*\\
        39 &   & if\ w\equiv 241^*\\
        41 &   & if\ w\equiv 145^*\\
        43 &   & if\ w\equiv 57^*\\
        45 &   & if\ w\equiv 233^*\\
        47 &   & if\ w\equiv 161^*\\
        49 &   & if\ w\equiv 97^*\\
        51 &   & if\ w\equiv 41^*\\
        53 &   & if\ w\equiv 249^*\\
        55 &   & if\ w\equiv 209^*\\
        57 &   & if\ w\equiv 177^*\\
        59 &   & if\ w\equiv 153^*\\
        61 &   & if\ w\equiv 137^*\\
        63 &   & if\ w\equiv 129^*\\
        s_1s_2\dots s_ge(W) &  & if\ w=s_1^2s_2^2\dots s_g^2W,W\ square-free
    \end{array}
\right.$$
(Here, $w\equiv1^*$ means $w\equiv 1\pmod{256}$ and square-free, and so forth.)
Hence for $w\equiv1\pmod8$, $\sqrt{w}\equiv e(w)\pmod{64}$ in $\mathbb{Q}_2$.
(Note that there are two solutions of $x^2\equiv w\pmod{64}$ in $\mathbb{Q}_2$. However, by taking conjugations we can take either of them as the value
of $\sqrt{w}$ in $\mathbb{Q}_2$.)

If $K\in SS(2)$, i.e. $D\equiv 1\pmod8$ and $A(D+Bc)=t^2T$ with $T$ square-free and $T\equiv 1\pmod8$, 
then $K_\mathcal{P}$s are all isomorphic to $\mathbb{Q}_2$, and
$m\equiv (x_1+y_1e(D))+(x_2+y_2e(D))e(A(D+Be(D)))\pmod{16}$. 
(Note that $x_i$ and $y_i$s can be all odd making $m$ a twice. If that happens, we still have 
$\frac{m}{2}\equiv \frac{1}{2}(x_1+y_1e(D))+(x_2+y_2e(D))e(A(D+Be(D)))\pmod8$, where 
$\frac{1}{2}(x_1+y_1e(D))+(x_2+y_2e(D))e(A(D+Be(D)))$ is no longer a twice.)

\begin{lemma}
    If $D\equiv 1\pmod8$ and $A(D+Be(D))=t^2T$ with $T$ square-free and $T\equiv 1\pmod8$,
    then $(\frac{-1,m}{\mathcal{P}})=1$ for all dyadic spots $\mathcal{P}$, if and only if 
    $(x_1+y_1e(D))+(x_2+y_2e(D))e(A(D+Be(D)))$, $(x_1+y_1e(D))+(x_2+y_2e(D))e(A(D-Be(D)))$,
    $(x_1-y_1e(D))+(x_2-y_2e(D))e(A(D+Be(D)))$ and $(x_1-y_1e(D))+(x_2-y_2e(D))e(A(D-Be(D)))$
    are all $\equiv1\pmod4$, or $\equiv2\pmod8$.
\end{lemma}

Next we analyze SR(2) fields.

\begin{lemma}
    If $D\equiv 1\pmod8$ and $l\neq0$,
    then $(\frac{-1,m}{\mathcal{P}})=1$ for both dyadic spots $\mathcal{P}$, 
    if and only if $2|m-1$ if $2\nmid m$,
    and $2|\frac{m}{2}-1$ otherwise.
\end{lemma}

\begin{proof}
    Directly obtained from Lemma 1.2.
\end{proof}

Next we analyze SI(2) fields. First we calculate $\sqrt{N}\pmod{16}$ in $K_\mathcal{P}\simeq\mathbb{Q}_2(\sqrt{5})$
where $N\equiv5\pmod8$ square-free.

\begin{lemma}
    Define
    $$
    e(N)=\left\{
    \begin{array}{lll}
        \sqrt{5} &   & if\ N\equiv5\mod{32}\\ 
        \sqrt{5}(1+2^2+2^3) &   & if\ N\equiv13\mod{32}\\ 
        \sqrt{5}(1+2^3) &   & if\ N\equiv21\mod{32}\\ 
        \sqrt{5}(1+2^2)  &   & if\ N\equiv29\mod{32}\\ 
        s_1s_2\dots s_ge(W) &  & if\ w=s_1^2s_2^2\dots s_g^2W,W\ square-free
    \end{array}
\right.
    $$
    Then without loss of generality we can choose $\sqrt{N}\equiv e(N)\mod{16}$.
\end{lemma}

\begin{proof}
    Note that $\sqrt{N}\equiv\sqrt{5}\mod{2}$ since $(\sqrt{N})^2$ is a square, and
    for $L$ an integer, $(\sqrt{32L+N}-\sqrt{N})(\sqrt{32L+N}+\sqrt{N})$ is divided by 32,
    with one of the factor is exactly divided by $2^1$. So we can choose $16|\sqrt{32L+N}-\sqrt{N}$. 
    Since $2^5||(\sqrt{21}-9\sqrt{5})(\sqrt{21}+9\sqrt{5})$ we choose $\sqrt{21}\equiv9\sqrt{5}\pmod{32}$.
    The rest of proof is similar.
\end{proof}

Combining Lemma 2.7 and 1.2 we obtain that
\begin{lemma}
    If $D\equiv 1\pmod8$ and $A(D+Be(D))$ is a power of 2 multiplying $N\equiv5\pmod8$,
    then $(\frac{-1,m}{\mathcal{P}})$ for dyadic spots $\mathcal{P}$, if and only if the 2-free part of
    $x_1+y_1e(D)+(x_2+y_2e(D))e(A(D+Be(D)))$ is $\equiv1,3,\frac{\pm3+\sqrt{5}}{2}\pmod4$.
\end{lemma}

We conclude the algorithm determining whether $m$ in $K$ is a sum of two squares or not:

\begin{theorem}
    Let $K$ be a cyclic quartic field in Theorem 0.1, 
    $m=X+Y\theta$, $X=x_1+x_2\sqrt{D}$, $Y=y_1+y_2\sqrt{D}$, $x_1,x_2,y_1,y_2\in\mathbb{Z}$,
    $(x_1,x_2,y_1,y_2)=1$
    
    Then the necessary and sufficient condition that $m$ is a sum of two squares, is

    (1)$m$ is totally positive if $A>0$;

    (2)For every prime $p\equiv3\pmod4$ dividing $N_{K/\mathbb{Q}}(m)$, the following non-dyadic conditions hold:

    (2-1)For $p|D$, $p|y_1, p\nmid x_2$.

    (2-2)For $(\frac{D}{p})=1$ and $(\frac{A(D+Bc)}{p})=1$ where $c$ is an integer that $c^2\equiv D\pmod p$,
    condition $[m,p]$ holds (recall Lemma 2.3).

    (2-3)For $(\frac{D}{p})=1$ and $p|A$,  
    $p\nmid m^2$ and $v_p(N_{K/\mathbb{Q}}(m))$ is even.

    (3)The following dyadic conditions hold:

    (3-1)For $D\equiv 1\pmod8$ and $A(D+Be(D))=t^2T$ with $T$ square-free and $T\equiv 1\pmod8$, 
    $(x_1+y_1e(D))+(x_2+y_2e(D))e(A(D+Be(D)))$, $(x_1+y_1e(D))+(x_2+y_2e(D))e(A(D-Be(D)))$,
    $(x_1-y_1e(D))+(x_2-y_2e(D))e(A(D+Be(D)))$ and $(x_1-y_1e(D))+(x_2-y_2e(D))e(A(D-Be(D)))$
    are all $\equiv1\pmod4$, or $\equiv2\pmod8$.

    (3-2)For $D\equiv 1\pmod8$ and $l\neq0$,
    $2|m-1$ if $2\nmid D$ and $2\nmid m$, and $2|\frac{m}{2}-1$ otherwise.

    (3-3)For $D\equiv 1\pmod8$ and $A(D+Be(D))=t^2T$ with $T$ square-free and $T\equiv 5\pmod8$,
    the 2-free part of
    $x_1+y_1e(D)+(x_2+y_2e(D))e(A(D+Be(D)))$ is $\equiv1,3,\frac{\pm3\pm\sqrt{5}}{2}\pmod4$.

\end{theorem}

Finally, to discuss general cases, we need to start from the values of $(\frac{-1,P}{\mathcal{P}})$,
where $P=p_1p_2\dots p_\alpha$ the product of some distinct prime integers $\equiv3\pmod4$,
and $\mathcal{P}$ an arbitrary finite place with $p\equiv3\pmod4$.

If $K\in RR(p)$, then $v_\mathcal{P}(P)$ is always 0 or 4, making $(\frac{-1,P}{\mathcal{P}})$ always 1
where $\mathcal{P}|p$.

If $K\in IR(p)$, then $v_\mathcal{P}(P)$ is always 0 or 2, making $(\frac{-1,P}{\mathcal{P}})$ always 1
where $\mathcal{P}|p$.

If $K\in II(p)$, then $k_{\mathfrak{p}}\simeq\mathbb{Q}_p(\sqrt{-1})$ since $(\frac{-1}{p})=-1$,
thus -1 is already a square in $k_{\mathfrak{p}}$, hence $(\frac{-1,P}{\mathcal{P}})$ is always 1.

If $K\in SS(p)$, then $(\frac{-1,P}{\mathcal{P}})$ are all -1 if $p|P$, and 1 if not, where $\mathcal{P}|p$.

If $K\in SR(p)$, then $v_\mathcal{P}(P)$ is always 0 or 2, making $(\frac{-1,P}{\mathcal{P}})$ always 1
where $\mathcal{P}|p$.

If $K\in SI(p)$, then $K_{\mathcal{P}}\simeq\mathbb{Q}_p(\sqrt{-1})$ since $(\frac{-1}{p})=-1$,
thus $-1$ is already a square in $K_{\mathcal{P}}$, hence $(\frac{-1,P}{\mathcal{P}})$ is always 1.

If $p\nmid N_{K/\mathbb{Q}}(m)$ but $p|P$, then $(\frac{-1,P}{\mathcal{P}})=1$ if and only if 
$pO_K$ ramifies, i.e., $p|AD$, which implies $p|A$ since $p\equiv3\pmod4$.

Then discuss $(\frac{-1,P}{\mathcal{P}_2})$ if $P_2$ is dyadic. We will use Lemma 1.1 and 1.2.

If $2O_K$ ramifies, i.e., $l\neq0$ in Theorem 0.2, then by Lemma 1.2,  $(\frac{-1,P}{\mathcal{P}_2})=1$.

If $K\in II(2)$, by Hilbert Reciprocity Law, $(\frac{-1,P}{\mathcal{P}_2})=
\displaystyle\prod_{\mathcal{P}|P}(\frac{-1,P}{\mathcal{P}})=\displaystyle\prod_{\mathcal{P}|P\ and\ K\in SS(p)}(\frac{-1,P}{\mathcal{P}})
=1$, a power of $(-1)^4$.

If $K\in SS(2)$, then $(\frac{-1,P}{\mathcal{P}_2})$ are all 1 if $P\equiv1\pmod4$, i.e, $\alpha$ is even, and $-1$ if not.

If $K\in SI(2)$, then $(\frac{-1,P}{\mathcal{P}_2})$ are both 1 by Lemma 1.2, since $K_{\mathcal{P}_2}\simeq\mathbb{Q}_2(\sqrt{5})$.

Concluding all above arguments, we obtain that 

\begin{theorem}
    Let $K$ be a cyclic quartic field in Theorem 0.1, 
    $M=\lambda^2PQm$, $0\neq\lambda\in\mathbb{Q}$, $P$(resp. $Q$) is the product of $\alpha$(resp. $\beta$) distinct 
    primes of 3$\pmod4$(resp. 1 or 2$\pmod4$), and 
    $m=X+Y\theta$, $X=x_1+x_2\sqrt{D}$, $Y=y_1+y_2\sqrt{D}$, $x_1,x_2,y_1,y_2\in\mathbb{Z}$,
    $(x_1,x_2,y_1,y_2)=1$.
    Then the necessary and sufficient condition that $M$ is a sum of two squares, is

    Then the necessary and sufficient condition that $m$ is a sum of two squares, is

    (1)$m$ is totally positive if $A>0$;

    (2)For every prime $p\equiv3\pmod4$ dividing $N_{K/\mathbb{Q}}(m)$, the following non-dyadic conditions hold:

    (2-1)For $p|D$, $p|y_1, p\nmid x_2$.

    (2-2)For $(\frac{D}{p})=1$ and $(\frac{A(D+Bc)}{p})=1$ where $c$ is an integer that $c^2\equiv D\pmod p$,
    condition $[m,p]$ holds (recall Lemma 2.3).

    (2-3)For $(\frac{D}{p})=1$ and $p|A$,  
    $p\nmid m^2$ and $v_p(N_{K/\mathbb{Q}}(m))$ is even.

    (3)The following dyadic conditions hold:

    (3-1)For $D\equiv 1\pmod8$ and $A(D+Be(D))=t^2T$ with $T$ square-free and $T\equiv 1\pmod8$, 
    $(x_1+y_1e(D))+(x_2+y_2e(D))e(A(D+Be(D)))$, $(x_1+y_1e(D))+(x_2+y_2e(D))e(A(D-Be(D)))$,
    $(x_1-y_1e(D))+(x_2-y_2e(D))e(A(D+Be(D)))$ and $(x_1-y_1e(D))+(x_2-y_2e(D))e(A(D-Be(D)))$
    are all $\equiv1\pmod4$, at most multiplying 2, with $\alpha$ even; or none of them is, with $\alpha$ odd.

    (3-2)For $D\equiv 1\pmod8$ and $l\neq0$,
    $2|m-1$ if $2\nmid D$ and $2\nmid m$, and $2|\frac{m}{2}-1$ otherwise.

    (3-3)For $D\equiv 1\pmod8$ and $A(D+Be(D))=t^2T$ with $T$ square-free and $T\equiv 5\pmod8$,
    the 2-free part of
    $x_1+y_1e(D)+(x_2+y_2e(D))e(A(D+Be(D)))$ is $\equiv1,3,\frac{\pm3+\sqrt{5}}{2}\pmod4$.

    (4)For odd prime $p\nmid N_{K/\mathbb{Q}}(m)$ but $p|P$, $p|A$.
\end{theorem}

\begin{example}
    Let $S=-19-11\sqrt{5}+(1-3\sqrt{5})\sqrt{-2(5-2\sqrt{5})}$. We prove it a sum of two squares in $K=\mathbb{Q}(\sqrt{-2(5-2\sqrt{5})})$.
    One can calculate that $N_{K/\mathbb{Q}}(S)=2^4\times139921$, where 139921 a prime $\equiv1\pmod4$.
    So we only need to verify dyadic cases. However, $2O_{\mathbb{Q}(\sqrt{5})}$ is inert, so $K$ has only one dyadic place.
    Hence $x^2+y^2=S$ is solvable locally, therefore globally in $K$. Actually
    $$S=(\sqrt{5}+\frac{1-\sqrt{5}}{2}\sqrt{-2(5-2\sqrt{5})})^2+(1-(2+\sqrt{5})\sqrt{-2(5-2\sqrt{5})})^2.$$
\end{example}

\begin{example}
    Let $S=668-130\sqrt{17}-2(1+\sqrt{17})\sqrt{17-2\sqrt{17}}$. We prove it a sum of two squares in $K=\mathbb{Q}(\sqrt{17-2\sqrt{17}})$.
    We have $2^1||S$ and let $s=\frac{S}{2}$. Then $N_{K/\mathbb{Q}}(s)=1494272141$, a prime $\equiv1\pmod4$.
    Since $S$ is totally positive, we only need to verify dyadic cases. 
    $2O_K$ can be split into two prime ideals, each of which isomorphic to $\mathbb{Q}_2(\sqrt{-29})=
    \mathbb{Q}_2(\sqrt{3})$. Thus we only need to verify that $2|s-1$ by Lemma 1.2, which is true.
    Actually,
    $$S=(1-\sqrt{17}+3\sqrt{17-2\sqrt{17}})^2+(2+(\sqrt{17}-2)\sqrt{17-2\sqrt{17}})^2.$$
\end{example}

\begin{example}
    Let $S=-624+126\sqrt{17}-2(\sqrt{17}+1)\sqrt{-(17-2\sqrt{17})}$. We prove it a sum of two squares in $K=\mathbb{Q}(\sqrt{-(17-2\sqrt{17})})$.
    We have $2^1||S$ and let $s=\frac{S}{2}$. Then $N_{K/\mathbb{Q}}(s)=11^2\times53\times150961$,
    where 53 and 150961 are primes $\equiv1\pmod4$. Since $(\frac{17}{11})=1$,
    the 11-adic local field of $K$ at least contains a subfield isomorphic to $\mathbb{Q}_{11}(\sqrt{-1})$.
    Hence we only need to compute dyadic cases. $2O_k$ splits and $\mathfrak{p}O_K$ inerts
    , where $\mathfrak{p}$ is an arbitrary dyadic prime of $O_k$, implying that 
    both dyadic local fields of $K$ are isomorphic to 
    $\mathbb{Q}_2(\sqrt{5})$. So we just need to compute $s\pmod4$ by Lemma 1.2.
    By Lemma 2.8 and $1137-120\sqrt{5}\equiv1\pmod4$, we complete the proof.
    Actually, 
    $$S=(1-\sqrt{17}+3\sqrt{-(17-2\sqrt{17})})^2+(2+(\sqrt{17}-2)\sqrt{-(17-2\sqrt{17})})^2.$$
\end{example}

\section*{Acknowledgements}

The author is sincerely grateful for directions by Prof Hourong Qin, and support by National Natural Science Foundation of China (11971224).

\section*{References}

[1]K. Hardy, R. H. Hudson, D. Richman, K. S. Williams, N. M. Holtz, Calculation of the class numbers of imaginary cyclic quartic
fields, Math. Comp. 49 (1987) 615-620.

[2]O. O' Meara, Introduction to Quadratic Forms, Springer-Verlag, 1973.

[3]K. S. Williams, K. Hardy, C. Friesen, On the evaluation of the Legendre symbol $(\frac{A+B\sqrt{m}}{p})$, Acta Arithmetica XLV (1985),
255-272.

[4]H. Qin, The Sum Of Two Squares In A Quadratic Field, Communications in Algebra, 25:1 (1997), 177-184.

\end{document}